\documentclass[12pt]{article}

\usepackage{amsmath, epsfig, cite}
\usepackage{amssymb}
\usepackage{amsfonts}
\usepackage{latexsym}
\usepackage{amsthm}

\newtheorem{thm}{Theorem}[section]

\newtheorem{cor}[thm]{Corollary}

\newtheorem{lem}[thm]{Lemma}

\numberwithin{equation}{section}

\renewcommand{\thefootnote}{}

\begin{document}

\begin{center}
{\large\bf $q$-Supercongruences from\\[2mm] Gasper and Rahman's summation
formula
 \footnote{The work is supported by the National Natural Science Foundations of China (Nos. 12071103 and
11661032).}}
\end{center}

\renewcommand{\thefootnote}{$\dagger$}

\vskip 2mm \centerline{Chuanan Wei}
\begin{center}
{School of Biomedical Information and Engineering,\\ Hainan Medical University, Haikou 571199, China
\\Email address: weichuanan78@163.com}
\end{center}


\vskip 0.7cm \noindent{\bf Abstract.} In 2017, He [Proc. Amer. Math.
Soc. 145 (2017), 501--508] established two spuercongruences on
truncated hypergeometric series and further proposed two related
conjectures. Subsequently, Liu [Results Math. 72 (2017), 2057--2066]
extended He's formulas and confirmed the second conjecture. However,
the first conjecture is still open up to now. With the help of the
creative microscoping method and the Chinese remainder theorem for
coprime polynomials, we derive several $q$-supercongruences modulo
the fourth and fifth powers of a cyclotomic polynomial from Gasper
and Rahman's summation formula for basic hypergeometric series. As
conclusions, He's first conjecture is confirmed and a more general
form of He's second conjecture is proved.



\section{Introduction}
For any complex number $x$ and nonnegative integer $n$, define the
shifted-factorial as
\[(x)_{n}=\Gamma(x+n)/\Gamma(x),\]
where $\Gamma(x)$ is the Gamma function.
 In his second letter to Hardy on February 27, 1913, Ramanujan mentioned the
identity
\begin{equation}\label{ramanujan}
\sum_{k=0}^{\infty}(-1)^k(4k+1)\frac{(1/2)_k^5}{k!^5}=\frac{2}{\Gamma(3/4)^4}.
\end{equation}
Let $p$ be an odd prime throughout the paper and $\mathbb{Z}_p$
stand for the ring of all $p$-adic integers. Define Morita's
$p$-adic Gamma function (cf. \cite[Chapter 7]{Robert}) by
 \[\Gamma_{p}(0)=1\quad \text{and}\quad \Gamma_{p}(n)
=(-1)^n\prod_{\substack{1\leqslant k< n\\
p\nmid k}}k,\quad \text{when}\quad n\in\mathbb{Z}^{+}.\]
Noting $\mathbb{N}$ is a dense subset of $\mathbb{Z}_p$ associated with the $p$-adic norm $|\cdot|_p$, for each
$x\in\mathbb{Z}_p$, the definition of $p$-adic Gamma function can be extended as
 \[\Gamma_{p}(x)
=\lim_{\substack{n\in\mathbb{N}\\
|x-n|_p\to0}}\Gamma_{p}(n).\]

 In 1997, Van Hamme
\cite[(A.2)]{Hamme} conjectured an interesting $p$-adic analogue of
\eqref{ramanujan}:
\begin{equation}\label{hamme-b}
\sum_{k=0}^{(p-1)/2}(-1)^k(4k+1)\frac{(1/2)_k^5}{k!^5}\equiv
\begin{cases} \displaystyle -\frac{p}{\Gamma_p(3/4)^4}  \pmod{p^3}, &\text{if $p\equiv 1\pmod 4$,}\\[10pt]
 0\pmod{p^3}, &\text{if $p\equiv 3\pmod 4$.}
\end{cases}
\end{equation}
 Swisher \cite{Swisher} proved that \eqref{hamme-b} also holds
modulo $p^5$ for $p>5$ and $p\equiv 1\pmod 4$. Liu \cite{Liu-b}
showed that, for $p>3$ and $p\equiv 3\pmod 4$,
\begin{equation*}
\sum_{k=0}^{(p-1)/2}(-1)^k(4k+1)\frac{(1/2)_k^5}{k!^5}\equiv
-\frac{p^3}{16}\Gamma_p(1/4)^4\pmod{p^4}.
\end{equation*}

In 2017, He \cite{He} discovered the two supercongruences:
\begin{align}
&\sum_{k=0}^{(p-1)/2}(6k+1)\frac{(1/2)_k^3(1/4)_k}{k!^44^k}
\notag\\
&\quad\equiv
\begin{cases} \displaystyle (-1)^{\frac{p+3}{4}}p\Gamma_p(1/2)\Gamma_p(1/4)^2  \pmod{p^2}, &\text{if $p\equiv 1\pmod 4$,}\\[3pt]
 0\pmod{p^2}, &\text{if $p\equiv 3\pmod 4$,}
\end{cases}
\label{He-a}
\\[5pt]
&\sum_{k=0}^{p-1}(6k+1)\frac{(1/2)_k^3(1/4)_k^2}{k!^5}
\notag\\
&\quad\equiv
\begin{cases} \displaystyle -p\Gamma_p(1/4)^4  \pmod{p^2}, &\text{if $p\equiv 1\pmod 4$,}\\[3pt]
 0\pmod{p^2}, &\text{if $p\equiv 3\pmod 4$},
\end{cases}
\label{He-b}
\end{align}
and further proposed the following two relevant conjectures:
\begin{align}
&\sum_{k=0}^{p-1}(6k+1)\frac{(1/2)_k^3(1/4)_k}{k!^44^k}
\equiv0\pmod{p^4}\:\: \text{with}\:\: p\equiv 3\pmod 4,
\label{He-c}\\[3pt]
&\sum_{k=0}^{p-1}(6k+1)\frac{(1/2)_k^3(1/4)_k^2}{k!^5}
\equiv0\pmod{p^3}\:\: \text{with}\:\: p\equiv 3\pmod 4.
  \label{He-d}
\end{align}
Subsequently, Liu \cite{Liu-a} proved that \eqref{He-a} and
\eqref{He-b} are true modulo $p^3$ and so verified the truth of \eqref{He-d}. However, the conjecture \eqref{He-c} is still open
up to now.

 For any complex numbers $x$, $q$ and nonnegative integer $n$, define the $q$-shifted factorial
 to be
 \begin{equation*}
(x;q)_{\infty}=\prod_{k=1}^{\infty}(1-xq^k)\quad\text{and}\quad
(x;q)_n=\frac{(x;q)_{\infty}}{(xq^n;q)_{\infty}}.
 \end{equation*}
For simplicity, we also adopt the compact notation
\begin{equation*}
(x_1,x_2,\dots,x_m;q)_{n}=(x_1;q)_{n}(x_2;q)_{n}\cdots(x_m;q)_{n},
 \end{equation*}
where $m\in\mathbb{Z}^{+}$ and $n\in\mathbb{Z}^{+}\cup\{0,\infty\}.$
Following Gasper and Rahman \cite{Gasper}, the basic hypergeometric
series can be defined as
$$
_{r}\phi_{s}\left[\begin{array}{c}
a_1,a_2,\ldots,a_{r}\\
b_1,b_2,\ldots,b_{s}
\end{array};q,\, z
\right] =\sum_{k=0}^{\infty}\frac{(a_1,a_2,\ldots, a_{r};q)_k}
{(q,b_1,b_2,\ldots,b_{s};q)_k}\bigg\{(-1)^kq^{\binom{k}{2}}\bigg\}^{1+s-r}z^k.
$$
Then Gasper and Rahman's summation for basic hypergeometric series
(cf. \cite[Equation (3.8.12)]{Gasper}) can be stated as
\begin{align}
&\sum_{k=0}^{\infty}\frac{1-aq^{3k}}{1-a}\frac{(a,b,q/b;q)_k(d,f,a^2q/df;q^2)_k}{(q^2,aq^2/b,abq;q^2)_k(aq/d,aq/f,df/a;q)_k}q^k
\notag\\[3pt]
  &\quad
+\frac{(aq,f/a,b,q/b;q)_{\infty}(d,aq^2/df,fq^2/d,df^2q/a^2;q^2)_{\infty}}{(a/f,fq/a,aq/d,df/a;q)_{\infty}(aq^2/b,abq,fq/ab,bf/a;q^2)_{\infty}}
\notag\\[3pt]
  &\quad
\times{_3}\phi_{2}\left[\begin{array}{c}
f,bf/a,fq/ab\\[3pt]
fq^2/d,df^2q/a^2
\end{array};\, q^2, q^2 \right]
\notag\\[3pt]
  &\quad
=\frac{(aq,f/a;q)_{\infty}(aq^2/bd,abq/d,bdf/a,dfq/ab;q^2)_{\infty}}{(aq/d,df/a;q)_{\infty}(aq^2/b,abq,bf/a,fq/ab;q^2)_{\infty}}.
\label{GR}
\end{align}

All over the paper, let $[r]$ be the $q$-integer $(1-q^r)/(1-q)$ and
$\Phi_n(q)$ denote the $n$-th cyclotomic polynomial in $q$:
\begin{equation*}
\Phi_n(q)=\prod_{\substack{1\leqslant k\leqslant n\\
\gcd(k,n)=1}}(q-\zeta^k),
\end{equation*}
where $\zeta$ is an $n$-th primitive root of unity. Taking advantage
of the creative microscoping method recently introduced by Guo and
Zudilin \cite{GuoZu-a},  Guo \cite{Guo-a2} and Wang and Yue \cite
{WY} gave a $q$-analogue of \eqref{hamme-b}:  for any positive odd
integer $n$,
\begin{align*}
&\sum_{k=0}^{M}(-1)^k[4k+1]\frac{(q;q^2)_k^4(q^2;q^4)_k}{(q^2;q^2)_k^4(q^4;q^4)_k}q^k
\notag\\[5pt]
&\quad\equiv
\begin{cases} \displaystyle [n]\frac{(q^2;q^4)_{(n-1)/4}^2}{(q^4;q^4)_{(n-1)/4}^2}  \pmod{[n]\Phi_n(q)^2}, &\text{if $n\equiv 1\pmod 4$,}\\[15pt]
 0  \pmod{[n]\Phi_n(q)^2}, &\text{if $n\equiv 3\pmod 4$,}
\end{cases}
\end{align*}
where $M=(n-1)/2$ or $n-1$. Then it is extended to the modulo
$[n]\Phi_n(q)^4$ case by  Wei \cite{Wei-b}. There are more
$q$-analogues of supercongruences in the literature, we refer the
reader to
\cite{Guo-adb,GS,GS20c,GuoZu-b,LW,LP,Tauraso,WY-a,Wei-a,Zu19}.

Inspired by the work just mentioned, we shall establish the
following four theorems.

\begin{thm}\label{thm-a}
Let $n$ be a positive integer subject to $n\equiv 3\pmod 4$.
Then
\begin{align*}
\sum_{k=0}^{n-1}[6k+1]\frac{(q;q^2)_k^3(q;q^4)_k}{(q^2;q^2)_k(q^4;q^4)_k^3}q^{k^2+k}
\equiv0\pmod{[n]\Phi_n(q)^3}.
\end{align*}
\end{thm}

Choosing $n=p^r$ and then letting $q\to1$ in the above theorem, we
obtain the supercongruence.

\begin{cor}\label{cor-a}
Let $p$ be an odd prime and $r$ a positive integer satisfying
$p^r\equiv 3\pmod 4$. Then
\begin{align*}
\sum_{k=0}^{p^r-1}(6k+1)\frac{(\frac{1}{2})_k^3(\frac{1}{4})_k}{k!^44^k}
\equiv0\pmod{p^{r+3}}.
\end{align*}
\end{cor}

When $r=1$, Corollary \ref{cor-a} becomes  \eqref{He-c}. So He's
first conjecture is confirmed.

\begin{thm}\label{thm-b}
Let $n$ be a positive integer subject to $n\equiv 1\pmod 4$.
Then
\begin{align*}
&\sum_{k=0}^{M}[6k+1]\frac{(q;q^2)_k^3(q;q^4)_k}{(q^2;q^2)_k(q^4;q^4)_k^3}q^{k^2+k}
\\[3pt]&\quad
\equiv[n]q^{(1-n)/4}\frac{(q^2;q^4)_{(n-1)/4}}{(q^4;q^4)_{(n-1)/4}}\bigg\{1-[n]^2\sum_{j=1}^{(n-1)/4}\frac{q^{4j}}{[4j]^2}\bigg\}\pmod{[n]\Phi_n(q)^3}.
\end{align*}
where $M=(n-1)/2$ or $n-1$.
\end{thm}

Fixing $n=p^r$ and then letting $q\to1$ in the upper theorem, we get
the conclusion.

\begin{cor}\label{cor-b}
Let $p$ be an odd prime and $r$ a positive integer satisfying
$p^r\equiv 1\pmod 4$. Then
\begin{align*}
\sum_{k=0}^{(p^r-1)/2}(6k+1)\frac{(\frac{1}{2})_k^3(\frac{1}{4})_k}{k!^44^k}
\equiv
\frac{p^r}{16}\frac{(\frac{1}{2})_{(p^r-1)/4}}{(1)_{(p^r-1)/4}}\Big\{16-p^{2r}H_{(p^r-1)/4}^{(2)}\Big\}\pmod{p^{r+3}},
\end{align*}
where the harmonic numbers of order 2 are defined by
\[H_{m}^{(2)}
  =\sum_{k=1}^m\frac{1}{k^{2}}\quad\text{with}\quad m\in \mathbb{Z}^{+}.\]
\end{cor}

\begin{thm}\label{thm-c}
Let $n$ be a positive integer subject to $n\equiv 3\pmod 4$.
Then
\begin{align*}
&\sum_{k=0}^{n-1}[6k+1]\frac{(q;q^2)_k^2(q^2;q^4)_k(q;q^4)_k^2}{(q^2;q^2)_k^2(q^4;q^4)_k^3}q^{2k}
\notag\\[3pt]
&\quad\equiv[3n]\frac{q^{2n}(2-q^n)}{(1+q^n)^2}\frac{(q^2;q^4)_{(3n-1)/4}^2}{(q^4;q^4)_{(3n-1)/4}^2}\pmod{[n]\Phi_n(q)^4}.
\end{align*}
\end{thm}

Setting $n=p^r$ and then letting $q\to1$ in Theorem \ref{thm-c}, we
arrive at the result.

\begin{cor}\label{cor-c}
Let $p$ be an odd prime and $r$ a positive integer satisfying
$p^r\equiv 3\pmod 4$. Then
\begin{align*}
\sum_{k=0}^{p^r-1}(6k+1)\frac{(\frac{1}{2})_k^3(\frac{1}{4})_k^2}{k!^5}
\equiv\frac{3p^r}{4}\frac{(\frac{1}{2})_{(3p^r-1)/4}^2}{(1)_{(3p^r-1)/4}^2}\pmod{p^{r+4}}.
\end{align*}
\end{cor}

It is easy to understand that the $r=1$ case of Corollary
\ref{cor-c} is an extension of He's second conjecture \eqref{He-d}.

\begin{thm}\label{thm-d}
Let $n$ be a positive integer subject to $n\equiv 1\pmod 4$. Then
\begin{align*}
&\sum_{k=0}^{M}[6k+1]\frac{(q;q^2)_k^2(q^2;q^4)_k(q;q^4)_k^2}{(q^2;q^2)_k^2(q^4;q^4)_k^3}q^{2k}
\notag\\[3pt]
&\quad\equiv
[n]\frac{(q^2;q^4)_{(n-1)/4}^2}{(q^4;q^4)_{(n-1)/4}^2}\bigg\{1+[n]^2(2-q^n)\sum_{j=1}^{(n-1)/2}\frac{(-1)^{j+1}q^{2j}}{[2j]^2}\bigg\}\pmod{[n]\Phi_n(q)^4}.
\end{align*}
where $M=(n-1)/2$ or $n-1$.
\end{thm}

Taking $n=p^r$ and then letting $q\to1$ in Theorem \ref{thm-d}, we
are led to the formula.

\begin{cor}\label{cor-d}
Let $p$ be an odd prime and $r$ a positive integer satisfying
$p^r\equiv 1\pmod 4$. Then
\begin{align*}
&\sum_{k=0}^{(p^r-1)/2}(6k+1)\frac{(\frac{1}{2})_k^3(\frac{1}{4})_k^2}{k!^5}
\\[3pt]&\qquad \equiv
\frac{p^r}{8}\frac{(\frac{1}{2})_{(p^r-1)/4}^2}{(1)_{(p^r-1)/4}^2}\Big\{8+2p^{2r}H_{(p^r-1)/2}^{(2)}-p^{2r}H_{(p^r-1)/4}^{(2)}\Big\}\pmod{p^{r+4}}.
\end{align*}
\end{cor}

The rest of the paper is arranged as follows. By means of Gasper and
Rahman's summation for basic hypergeometric series, the creative
microscoping method, and the Chinese remainder theorem for coprime
polynomials, we shall deduce the parametric extensions of Theorems
\ref{thm-a} and \ref{thm-b} and then prove these two theorems in
Section 2. Theorems \ref{thm-c} and \ref{thm-d} can similarly be
proved. The corresponding details are deferred to Section 3.

\section{Proof of Theorems \ref{thm-a} and \ref{thm-b}}
For the sake of proving Theorems \ref{thm-a} and \ref{thm-b}, we
require the following Lemma.

\begin{lem}\label{lemm-a}
Let $n$ be a positive odd integer. Then
\begin{align}\label{eq:wei-a}
\sum_{k=0}^{M}[6k+1]\frac{(q,aq,q/a;q^2)_k(q/b;q^4)_k}{(q^4,q^4/a,aq^4;q^4)_k(bq^2;q^2)_k}q^{k^2+k}b^k
\equiv0\pmod{[n]}.
\end{align}
where $M=(n-1)/2$ or $n-1$.
\end{lem}

\begin{proof}
Above all, it is ordinary to realize that the $n=1$ case of Lemma
\ref{lemm-a} is right. Afterwards, we shall consider the $n>1$ case.
Setting $d= q^{-2n}$ and then letting $n\to\infty$ in \eqref{GR}, we
have
\begin{align}\label{GR-a}
&\sum_{k=0}^{\infty}\frac{1-aq^{3k}}{1-a}\frac{(a,b,q/b;q)_k(f;q^2)_k}{(q^2,aq^2/b,abq;q^2)_k(aq/f;q)_k}q^{\frac{k^2+k}{2}}\bigg(\frac{a}{f}\bigg)^k
\notag\\[3pt]
&\quad=\frac{(aq,aq^2,aq^2/bf,abq/f;q^2)_{\infty}}{(aq/f,aq^2/f,aq^2/b,abq;q^2)_{\infty}}.
\end{align}
The case $a\to q^{1-n}$, $b\to aq$, $f\to q/b$, $q\to q^2$ of it
gives
\begin{align*}
\sum_{k=0}^{M}\frac{1-q^{1+6k-n}}{1-q^{1-n}}\frac{(q^{1-n},aq,q/a;q^2)_k(q/b;q^4)_k}{(q^4,q^{4-n}/a,aq^{4-n};q^4)_k(bq^{2-n};q^2)_k}q^{k^2+k}(bq^{-n})^k
=0.
\end{align*}
On account of $q^n\equiv1\pmod{\Phi_n(q)}$, we find
\begin{align}\label{eq:wei-b}
\sum_{k=0}^{M}[6k+1]\frac{(q,aq,q/a;q^2)_k(q/b;q^4)_k}{(q^4,q^4/a,aq^4;q^4)_k(bq^2;q^2)_k}q^{k^2+k}b^k
\equiv0\pmod{\Phi_n(q)}.
\end{align}

Let $\alpha_q(k)$ stand for the $k$-th term on the left-hand side of
\eqref{eq:wei-b}, i.e.,
\begin{align*}
\alpha_q(k)=[6k+1]\frac{(q,aq,q/a;q^2)_k(q/b;q^4)_k}{(q^4,q^4/a,aq^4;q^4)_k(bq^2;q^2)_k}q^{k^2+k}b^k.
\end{align*}
Let $\zeta\neq1$ be an $n$-th root of unity, which is not
necessarily primitive. This implies that $\zeta$ is a primitive root
of unity of odd degree $m|n$. The $q$-congruence \eqref{eq:wei-b}
with $n=m$ shows that
\begin{align*}
\sum_{k=0}^{m-1}\alpha_{\zeta}(k)=\sum_{k=0}^{(m-1)/2}\alpha_{\zeta}(k)=0.
\end{align*}
In terms of the relation:
\begin{align*}
\frac{\alpha_{\zeta}(jm+k)}{\alpha_{\zeta}(jm)}=\lim_{q\to\zeta}\frac{\alpha_{q}(jm+k)}{\alpha_{q}(jm)}=\alpha_{\zeta}(k),
\end{align*}
There holds
\begin{align*}
\sum_{k=0}^{n-1}\alpha_{\zeta}(k)=\sum_{j=0}^{n/m-1}\sum_{k=0}^{m-1}\alpha_{\zeta}(jm+k)
=\sum_{j=0}^{n/m-1}\alpha_{\zeta}(jm)\sum_{k=0}^{m-1}\alpha_{\zeta}(k)=0,
\end{align*}
\begin{align*}
\sum_{k=0}^{(n-1)/2}\alpha_{\zeta}(k)=
\sum_{j=0}^{(n/m-3)/2}\alpha_{\zeta}(jm)\sum_{k=0}^{m-1}\alpha_{\zeta}(k)+\sum_{k=0}^{(m-1)/2}\alpha_{\zeta}((n-m)/2+k)=0.
\end{align*}
The last two equations indicate that $\sum_{k=0}^{n-1}\alpha_{q}(k)$
and $\sum_{k=0}^{(n-1)/2}\alpha_{q}(k)$ are both divisible by the
cyclotomic polynomials $\Phi_m(q)$. Since this is correct for any
divisor $m>1$ of $n$, we can point out that they are divisible by
\begin{equation*}
\prod_{m|n,m>1}\Phi_m(q)=[n].
\end{equation*}
Therefore, we complete the proof of Lemma \ref{lemm-a}.

\end{proof}

Now we display a parametric extension of Theorem \ref{thm-a}.

\begin{thm}\label{thm-e}
Let $n$ be a positive integer subject to $n\equiv 3\pmod 4$.
Then, modulo $[n](1-aq^n)(a-q^n)(b-q^n)$,
\begin{align}
&\sum_{k=0}^{n-1}[6k+1]\frac{(q,aq,q/a;q^2)_k(q/b^3;q^4)_k}{(q^4,q^4/a,aq^4;q^4)_k(b^3q^2;q^2)_k}q^{k^2+k}b^{3k}
\notag\\[3pt]
 &\quad\equiv
 \frac{(b-q^n)(ab-1-a^2+aq^n)}{(a-b)(1-ab)}\frac{(q^3,1/b^3;q^4)_{(n+1)/4}}{(1/q,b^3q^2;q^4)_{(n+1)/4}}\bigg(\frac{b^3}{q}\bigg)^{\frac{n+1}{4}}
\notag\\[3pt]
&\qquad+\frac{(1-aq^n)(a-q^n)}{(a-b)(1-ab)}\frac{(q^3,q^5;q^4)_{(3n-1)/4}}{(aq^4,q^4/a;q^2)_{(3n-1)/4}}.
\label{eq:wei-c}
\end{align}
 \end{thm}

\begin{proof}
When $a=q^{-n}$ or $a=q^n$, the left-hand side of  \eqref{eq:wei-c}
equals
\begin{align}
\sum_{k=0}^{n-1}[6k+1]\frac{(q,q^{1-n},q^{1+n};q^2)_k(q/b^3;q^4)_k}{(q^4,q^{4+n},q^{4-n};q^4)_k(b^3q^2;q^2)_k}q^{k^2+k}b^{3k}.
\label{eq:wei-d}
\end{align}
Via \eqref{GR-a}, the series \eqref{eq:wei-d} can be restated as
\begin{align*}
\frac{(q^3,1/b^3;q^4)_{(n+1)/4}}{(1/q,b^3q^2;q^4)_{(n+1)/4}}\bigg(\frac{b^3}{q}\bigg)^{\frac{n+1}{4}}.
\end{align*}
Because $(1-aq^n)$ and $(a-q^n)$ are pairwise relatively prime
polynomials, we discover the conclusion: modulo $(1-aq^n)(a-q^n)$,
\begin{align}
\sum_{k=0}^{n-1}[6k+1]\frac{(q,aq,q/a;q^2)_k(q/b^3;q^4)_k}{(q^4,q^4/a,aq^4;q^4)_k(b^3q^2;q^2)_k}q^{k^2+k}b^{3k}
\equiv
 \frac{(q^3,1/b^3;q^4)_{(n+1)/4}}{(1/q,b^3q^2;q^4)_{(n+1)/4}}\bigg(\frac{b^3}{q}\bigg)^{\frac{n+1}{4}}.
\label{eq:wei-f}
\end{align}

When $b=q^{n}$, the left-hand side of  \eqref{eq:wei-c} is equal to
\begin{align}
\sum_{k=0}^{n-1}[6k+1]\frac{(q,aq,q/a;q^2)_k(q^{1-3n};q^4)_k}{(q^4,q^4/a,aq^4;q^4)_k(q^{2+3n};q^2)_k}q^{k^2+k+3nk}.
\label{eq:wei-g}
\end{align}
Through \eqref{GR-a}, the series \eqref{eq:wei-g} can be rewritten
as
\begin{align*}
\frac{(q^3,q^5;q^4)_{(3n-1)/4}}{(aq^4,q^4/a;q^2)_{(3n-1)/4}}.
\end{align*}
So we there is the result: modulo $(b-q^n)$,
\begin{align}
\sum_{k=0}^{n-1}[6k+1]\frac{(q,aq,q/a;q^2)_k(q/b^3;q^4)_k}{(q^4,q^4/a,aq^4;q^4)_k(b^3q^2;q^2)_k}q^{k^2+k}b^{3k}
\equiv
 \frac{(q^3,q^5;q^4)_{(3n-1)/4}}{(aq^4,q^4/a;q^2)_{(3n-1)/4}}.
\label{eq:wei-h}
\end{align}

It is clear that the polynomials $(1-aq^n)(a-q^n)$, $(b-q^n)$, and
$[n]$ are relatively prime to one another. Noting the relations
\begin{align}
&\frac{(b-q^n)(ab-1-a^2+aq^n)}{(a-b)(1-ab)}\equiv1\pmod{(1-aq^n)(a-q^n)},
\label{Chinese-a}
\\[5pt]
\label{Chinese-b}
&\qquad\qquad\frac{(1-aq^n)(a-q^n)}{(a-b)(1-ab)}\equiv1\pmod{(b-q^n)},
\end{align}
and employing the Chinese remainder theorem for coprime polynomials,
we can derive Theorem \ref{thm-e} from Lemma \ref{lemm-a},
\eqref{eq:wei-f}, and \eqref{eq:wei-h}.
\end{proof}

\begin{proof}[Proof of Theorem \ref{thm-a}]
The $b\to1$ case of Theorem \ref{thm-e} produces the formula: modulo
$[n]\Phi_n(q)(1-aq^n)(a-q^n)$,
\begin{align*}
&\sum_{k=0}^{n-1}[6k+1]\frac{(aq,q/a;q^2)_k(q;q^2)_k(q;q^4)_k}{(q^4/a,aq^4;q^4)_k(q^2;q^2)_k(q^4;q^4)_k}q^{k^2+k}\\
 &\quad\equiv
\frac{(1-aq^n)(a-q^n)}{-(1-a)^2}\frac{(q^3,q^5;q^4)_{(3n-1)/4}}{(aq^4,q^4/a;q^2)_{(3n-1)/4}}.
\end{align*}
Considering that there is the factor $(1-q^n)^2$ in the numerator
$(q^3,q^5;q^4)_{(3n-1)/4}$, we obtain the $q$-supercongruence:
\begin{align*}
&\sum_{k=0}^{n-1}[6k+1]\frac{(aq,q/a;q^2)_k(q;q^2)_k(q;q^4)_k}{(q^4/a,aq^4;q^4)_k(q^2;q^2)_k(q^4;q^4)_k}q^{k^2+k}\equiv0\pmod{[n]\Phi_n(q)(1-aq^n)(a-q^n)}.
\end{align*}
When $q\to1$, it reduces to Theorem \ref{thm-a} exactly.

Next, we shall provide a parametric extension of Theorem
\ref{thm-b}.

\begin{thm}\label{thm-f}
Let $n$ be a positive integer subject to $n\equiv 1\pmod 4$.
Then, modulo $[n](1-aq^n)(a-q^n)(b-q^n)$,
\begin{align}
&\sum_{k=0}^{M}[6k+1]\frac{(q,aq,q/a;q^2)_k(q/b;q^4)_k}{(q^4,q^4/a,aq^4;q^4)_k(bq^2;q^2)_k}q^{k^2+k}b^{k}
\notag\\[3pt]
 &\quad\equiv\frac{(b-q^n)(ab-1-a^2+aq^n)}{(a-b)(1-ab)}\frac{(q^5,q^2/b;q^4)_{(n-1)/4}}{(q,bq^4;q^4)_{(n-1)/4}}\bigg(\frac{b}{q}\bigg)^{\frac{n-1}{4}}
\notag\\[3pt]
&\qquad+\frac{(1-aq^n)(a-q^n)}{(a-b)(1-ab)}\frac{(q^3,q^5;q^4)_{(n-1)/4}}{(aq^4,q^4/a;q^4)_{(n-1)/4}},
\label{eq:wei-i}
\end{align}
where $M=(n-1)/2$ or $n-1$.
 \end{thm}
\end{proof}

\begin{proof}
The $q\to q^2$, $a\to q$, $b\to q^{1-n}$, $f\to q/b$ case of
\eqref{GR-a} reads
\begin{align*}
\sum_{k=0}^{M}[6k+1]\frac{(q,q^{1-n},q^{1+n};q^2)_k(q/b;q^4)_k}{(q^4,q^{4+n},q^{4-n};q^4)_k(bq^2;q^2)_k}q^{k^2+k}b^{k}
=\frac{(q^5,q^2/b;q^4)_{(n-1)/4}}{(q,bq^4;q^4)_{(n-1)/4}}\bigg(\frac{b}{q}\bigg)^{\frac{n-1}{4}}.
\end{align*}
Thus we get hold of the conclusion: modulo $(1-aq^n)(a-q^n)$,
\begin{align}
\sum_{k=0}^{M}[6k+1]\frac{(q,aq,q/a;q^2)_k(q/b;q^4)_k}{(q^4,q^4/a,aq^4;q^4)_k(bq^2;q^2)_k}q^{k^2+k}b^{k}
\equiv
 \frac{(q^5,q^2/b;q^4)_{(n-1)/4}}{(q,bq^4;q^4)_{(n-1)/4}}\bigg(\frac{b}{q}\bigg)^{\frac{n-1}{4}}.
\label{eq:wei-j}
\end{align}

The $q\to q^2$, $a\to q$, $b\to aq$, $f\to q^{1-n}$  case of
\eqref{GR-a} is
\begin{align*}
\sum_{k=0}^{M}[6k+1]\frac{(q,aq,q/a;q^2)_k(q^{1-n};q^4)_k}{(q^4,q^{4}/a,aq^{4};q^4)_k(q^{2+n};q^2)_k}q^{k^2+k+nk}
=\frac{(q^3,q^5;q^4)_{(n-1)/4}}{(aq^4,q^4/a;q^4)_{(n-1)/4}}.
\end{align*}
 Hence we arrive at the result: modulo
$(b-q^n)$,
\begin{align}
\sum_{k=0}^{M}[6k+1]\frac{(q,aq,q/a;q^2)_k(q/b;q^4)_k}{(q^4,q^4/a,aq^4;q^4)_k(bq^2;q^2)_k}q^{k^2+k}b^{k}
\equiv \frac{(q^3,q^5;q^4)_{(n-1)/4}}{(aq^4,q^4/a;q^4)_{(n-1)/4}}.
\label{eq:wei-k}
\end{align}
Using \eqref{Chinese-a}, \eqref{Chinese-b}, and the Chinese
remainder theorem for coprime polynomials, we can prove Theorem
\ref{thm-f} according to Lemma \ref{lemm-a}, \eqref{eq:wei-j}, and
\eqref{eq:wei-k}.
\end{proof}

\begin{proof}[Proof of Theorem \ref{thm-b}]
Letting $b\to1$ in Theorem \ref{thm-f}, we can show that, modulo
$[n]\Phi_n(q)(1-aq^n)(a-q^n)$,
\begin{align}
&\sum_{k=0}^{M}[6k+1]\frac{(aq,q/a;q^2)_k(q;q^2)_k(q;q^4)_k}{(q^4/a,aq^4;q^4)_k(q^2;q^2)_k(q^4;q^4)_k}q^{k^2+k}
\notag\\[3pt]\notag
 &\quad\equiv
[n]\frac{(q^2;q^4)_{(3n-1)/4}}{(q^4;q^4)_{(n-1)/4}}q^{\frac{1-n}{4}}\\[3pt]
&\qquad+\frac{(1-aq^n)(a-q^n)}{(1-a)^2}\bigg\{\frac{(q^5,q^2;q^4)_{(n-1)/4}}{(q,q^4;q^4)_{(n-1)/4}}q^{\frac{1-n}{4}}
-\frac{(q^5,q^3;q^4)_{(n-1)/4}}{(aq^4,q^4/a;q^4)_{(n-1)/4}}\bigg\}
\notag\\[3pt]\notag &\quad\equiv
[n]\frac{(q^2;q^4)_{(3n-1)/4}}{(q^4;q^4)_{(n-1)/4}}q^{\frac{1-n}{4}}\\[3pt]
&\qquad+\frac{(1-aq^n)(a-q^n)}{(1-a)^2}\bigg\{\frac{(q^5,q^3;q^4)_{(n-1)/4}}{(q^4,q^4;q^4)_{(n-1)/4}}
-\frac{(q^5,q^3;q^4)_{(n-1)/4}}{(aq^4,q^4/a;q^4)_{(n-1)/4}}\bigg\}.
\label{eq:wei-l}
\end{align}
 By the L'H\^{o}spital rule, it is not difficult to evaluate the
 limit
\begin{align*}
&\lim_{a\to1}\frac{1}{(1-a)^2}\bigg\{\frac{(q^5,q^3;q^4)_{(n-1)/4}}{(q^4,q^4;q^4)_{(n-1)/4}}
-\frac{(q^5,q^3;q^4)_{(n-1)/4}}{(aq^4,q^4/a;q^4)_{(n-1)/4}}\bigg\}\\[5pt]
&\quad=-\frac{(q^5,q^3;q^4)_{(n-1)/4}}{(q^4,q^4;q^4)_{(n-1)/4}}
\sum_{j=1}^{(n-1)/4}\frac{q^{4j}}{(1-q^{4j})^2}.
\end{align*}
Letting $a\to1$ in \eqref{eq:wei-l} and utilizing the above limit,
we catch hold of the $q$-supercongruence: modulo $[n]\Phi_n(q)^3$,
\begin{align*}
&\sum_{k=0}^{M}[6k+1]\frac{(q;q^2)_k^3(q;q^4)_k}{(q^2;q^2)_k(q^4;q^4)_k^3}q^{k^2+k}
\\[3pt]&\quad
\equiv[n]\frac{(q^2;q^4)_{(3n-1)/4}}{(q^4;q^4)_{(n-1)/4}}q^{\frac{1-n}{4}}-[n]^2\frac{(q^5,q^3;q^4)_{(n-1)/4}}{(q^4,q^4;q^4)_{(n-1)/4}}
\sum_{j=1}^{(n-1)/4}\frac{q^{4j}}{[4j]^2}
\\[3pt]&\quad
\equiv[n]q^{(1-n)/4}\frac{(q^2;q^4)_{(n-1)/4}}{(q^4;q^4)_{(n-1)/4}}\bigg\{1-[n]^2\sum_{j=1}^{(n-1)/4}\frac{q^{4j}}{[4j]^2}\bigg\}.
\end{align*}
This completes the proof of Theorem \ref{thm-b}.
\end{proof}
\section{Proof of Theorems \ref{thm-c} and \ref{thm-d}}
In order to prove Theorems \ref{thm-c} and \ref{thm-d}, we need the
following lemma.

\begin{lem}\label{lemm-b}
Let $n$ be a positive odd integer. Then
\begin{align*}
\sum_{k=0}^{M}[6k+1]\frac{(aq,q/a;q^2)_k(q^2,bq,q/b;q^4)_k}{(q^4,q^4/a,aq^4;q^4)_k(q^2/b,bq^2;q^2)_k}q^{2k}
\equiv0\pmod{[n]}.
\end{align*}
where $M=(n-1)/2$ or $n-1$.
\end{lem}

\begin{proof}
On one hand, it is routine to verify the correctness of the $n=1$
case of Lemma \ref{lemm-b}. On the other hand, we shall discuss the
$n>1$ case. Letting $a\to q^{1-n}$, $b\to aq$, $d\to q^{2-2n}$,
$f\to bq$, $q\to q^2$ in \eqref{GR}, we have
\begin{align*}
\sum_{k=0}^{M}\frac{1-q^{1+6k-n}}{1-q^{1-n}}\frac{(q^{1-n},aq,q/a;q^2)_k(q^{2-2n},bq,q/b;q^4)_k}{(q^4,q^{4-n}/a,aq^{4-n};q^4)_k(q^{1+n},q^{2-n}/b,bq^{2-n};q^2)_k}q^{2k}
=0.
\end{align*}
Thanks to $q^n\equiv1\pmod{\Phi_n(q)}$, we find
\begin{align*}
\sum_{k=0}^{M}[6k+1]\frac{(aq,q/a;q^2)_k(q^2,bq,q/b;q^4)_k}{(q^4,q^4/a,aq^4;q^4)_k(q^2/b,bq^2;q^2)_k}q^{2k}
\equiv0\pmod{\Phi_n(q)}.
\end{align*}
Similar to the proof of Lemma \ref{lemm-a}, we can deduce Lemma
\ref{lemm-b} from the last relation.
\end{proof}

Now we shall give a parametric extension of Theorem \ref{thm-c}.

\begin{thm}\label{thm-g}
Let $n$ be a positive integer subject to $n\equiv 3\pmod 4$.
Then, modulo $[n](1-aq^n)(a-q^n)(1-bq^n)(b-q^n)$,
\begin{align}
&\sum_{k=0}^{n-1}[6k+1]\frac{(aq,q/a;q^2)_k(q^2,b^3q,q/b^3;q^4)_k}{(q^4,q^4/a,aq^4;q^4)_k(q^2/b^3,b^3q^2;q^2)_k}q^{2k}
\notag\\[3pt]
 &\quad\equiv\frac{(1-aq^n)(a-q^n)(-1-b^2+bq^n)}{(b-a)(1-ba)}\frac{(q^5,aq^2,q^2/a;q^4)_{(3n-1)/4}}{(q,q^4/a,aq^4;q^4)_{(3n-1)/4}}
\notag\\[3pt]
&\qquad+\frac{(1-bq^n)(b-q^n)(-1-a^2+aq^n)}{(a-b)(1-ab)}\frac{(q^3,b^3,1/b^3;q^4)_{(n+1)/4}}{(1/q,q^2/b^3,b^3q^2;q^4)_{(n+1)/4}}.
\label{eq:wei-n}
\end{align}
 \end{thm}

\begin{proof}
When $a=q^{-n}$ or $a=q^n$, the left-hand side of  \eqref{eq:wei-n}
equals
\begin{align}
\sum_{k=0}^{n-1}[6k+1]\frac{(q^{1-n},q^{1+n};q^2)_k(q^2,b^3q,q/b^3;q^4)_k}{(q^4,q^{4+n},q^{4-n};q^4)_k(q^2/b^3,b^3q^2;q^2)_k}q^{2k}.
\label{eq:wei-o}
\end{align}
Via \eqref{GR}, the series \eqref{eq:wei-o} can be expressed as
\begin{align*}
\frac{(q^3,b^3,1/b^3;q^4)_{(n+1)/4}}{(1/q,q^2/b^3,b^3q^2;q^4)_{(n+1)/4}}.
\end{align*}
Since $(1-aq^n)$ and $(a-q^n)$ are pairwise relatively prime
polynomials, there holds the $q$-congruence: modulo
$(1-aq^n)(a-q^n)$,
\begin{align}
\sum_{k=0}^{n-1}[6k+1]\frac{(aq,q/a;q^2)_k(q^2,b^3q,q/b^3;q^4)_k}{(q^4,q^4/a,aq^4;q^4)_k(q^2/b^3,b^3q^2;q^2)_k}q^{2k}
\equiv
 \frac{(q^3,b^3,1/b^3;q^4)_{(n+1)/4}}{(1/q,q^2/b^3,b^3q^2;q^4)_{(n+1)/4}}.
\label{eq:wei-p}
\end{align}

When $b=q^{-n}$ or $b=q^n$, the left-hand side of  \eqref{eq:wei-n}
is equal to
\begin{align}
\sum_{k=0}^{n-1}[6k+1]\frac{(aq,q/a;q^2)_k(q^2,q^{1-3n},q^{1+3n};q^4)_k}{(q^4,q^4/a,aq^4;q^4)_k(q^{2+3n},q^{2-3n};q^2)_k}q^{2k}
\label{eq:wei-q}
\end{align}
Through \eqref{GR}, the series \eqref{eq:wei-q} can be modified as
\begin{align*}
\frac{(q^5,aq^2,q^2/a;q^4)_{(3n-1)/4}}{(q,q^4/a,aq^4;q^4)_{(3n-1)/4}}.
\end{align*}
Therefore, there is the result: modulo $(1-bq^n)(b-q^n)$,
\begin{align}
\sum_{k=0}^{n-1}[6k+1]\frac{(aq,q/a;q^2)_k(q^2,b^3q,q/b^3;q^4)_k}{(q^4,q^4/a,aq^4;q^4)_k(q^2/b^3,b^3q^2;q^2)_k}q^{2k}
\equiv
\frac{(q^5,aq^2,q^2/a;q^4)_{(3n-1)/4}}{(q,q^4/a,aq^4;q^4)_{(3n-1)/4}}.
\label{eq:wei-r}
\end{align}

It is clear that the polynomials $(1-aq^n)(a-q^n)$,
$(1-aq^n)(a-q^n)$, and $[n]$ are prime to each other. Due to the
relations
\begin{align}
&\frac{(1-bq^n)(b-q^n)(-1-a^2+aq^n)}{(a-b)(1-ab)}\equiv1\pmod{(1-aq^n)(a-q^n)},
\label{Chinese-c}
\\[5pt]
&\frac{(1-aq^n)(a-q^n)(-1-b^2+bq^n)}{(b-a)(1-ba)}\equiv1\pmod{(1-bq^n)(b-q^n)},
\label{Chinese-d}
\end{align}
and employing the Chinese remainder theorem for coprime polynomials,
we can establish Theorem \ref{thm-g} in accordance with Lemma
\ref{lemm-b}, \eqref{eq:wei-p}, and \eqref{eq:wei-r}.
\end{proof}

\begin{proof}[Proof of Theorem \ref{thm-c}]
Letting $b\to1$ in Theorem \ref{thm-g}, we discover the formula:
modulo $[n]\Phi_n(q)^2(1-aq^n)(a-q^n)$,
\begin{align}
&\sum_{k=0}^{n-1}[6k+1]\frac{(aq,q/a;q^2)_k(q;q^4)_k^2(q^2;q^4)_k}{(q^4/a,aq^4;q^4)_k(q^2;q^2)_k^3(q^4;q^4)_k}q^{2k}
\notag\\[3pt]
 &\quad\equiv
[3n]\frac{(1-aq^n)(a-q^n)(-2+q^n)}{(1-a)^2}\frac{(aq^2,q^2/a;q^4)_{(3n-1)/4}}{(q^4/a,aq^4;q^4)_{(3n-1)/4}}.
\label{eq:wei-s}
\end{align}
It is easy to see the relation: modulo $\Phi_n(q)^2$,
\begin{align}
&\frac{(aq^2,q^2/a;q^4)_{(3n-1)/4}}{(q^4/a,aq^4;q^4)_{(3n-1)/4}}
\notag\\[3pt]
&=q^{2n}(1-a)(1-1/a)\frac{(aq^2,q^2/a;q^4)_{(n-1)/2}(aq^{2n+4},q^{2n+4}/a;q^4)_{(n-3)/4}}{(q^4/a,aq^4;q^4)_{(3n-1)/4}}.
\label{eq:wei-t}
\end{align}
The combination \eqref{eq:wei-s} and \eqref{eq:wei-t} produces the
$q$-supercnogruence: modulo $[n]\Phi_n(q)^2(1-aq^n)(a-q^n)$,
\begin{align}
&\sum_{k=0}^{n-1}[6k+1]\frac{(aq,q/a;q^2)_k(q;q^4)_k^2(q^2;q^4)_k}{(q^4/a,aq^4;q^4)_k(q^2;q^2)_k^3(q^4;q^4)_k}q^{2k}
\notag\\[3mm]
 &\quad\equiv
[3n](1-aq^n)(a-q^n)(2-q^n)(q^{2n}/a)
\notag\\[3mm]
&\qquad\times\frac{(aq^2,q^2/a;q^4)_{(n-1)/2}(aq^{2n+4},q^{2n+4}/a;q^4)_{(n-3)/4}}{(q^4/a,aq^4;q^4)_{(3n-1)/4}}.
\label{eq:wei-u}
\end{align}
Letting $a\to1$ in \eqref{eq:wei-u}, we catch hold of Theorem
\ref{thm-c}.
\end{proof}

Next, we shall offer a parametric extension of Theorem \ref{thm-d}.

\begin{thm}\label{thm-h}
Let $n$ be a positive integer subject to $n\equiv 1\pmod 4$.
Then, modulo modulo $[n](1-aq^n)(a-q^n)(1-bq^n)(b-q^n)$,
\begin{align}
&\sum_{k=0}^{M}[6k+1]\frac{(aq,q/a;q^2)_k(q^2,bq,q/b;q^4)_k}{(q^4,q^4/a,aq^4;q^4)_k(q^2/b,bq^2;q^2)_k}q^{2k}
\notag\\[3pt]
 &\quad\equiv\frac{(1-aq^n)(a-q^n)(-1-b^2+bq^n)}{(b-a)(1-ba)}\frac{(q^5,aq^2,q^2/a;q^4)_{(n-1)/4}}{(q,q^4/a,aq^4;q^4)_{(n-1)/4}}
\notag\\[3pt]
&\qquad+\frac{(1-bq^n)(b-q^n)(-1-a^2+aq^n)}{(a-b)(1-ab)}\frac{(q^5,bq^2,q^2/b;q^4)_{(n-1)/4}}{(q,q^4/b,bq^4;q^4)_{(n-1)/4}},
\label{eq:wei-v}
\end{align}
where $M=(n-1)/2$ or $n-1$.
 \end{thm}

\begin{proof}
The $q\to q^2$, $a\to q$, $b\to q^{1-n}$, $d\to bq$, $f\to q^2$ case
of \eqref{GR} reads
\begin{align*}
\sum_{k=0}^{M}[6k+1]\frac{(q^{1-n},q^{1+n};q^2)_k(q^2,bq,q/b;q^4)_k}{(q^4,q^{4+n},q^{4-n};q^4)_k(q^2/b,bq^2;q^2)_k}q^{2k}
=\frac{(q^5,bq^2,q^2/b;q^4)_{(n-1)/4}}{(q,q^4/b,bq^4;q^4)_{(n-1)/4}}.
\end{align*}
So we arrive at conclusion: modulo $(1-aq^n)(a-q^n)$,
\begin{align}
\sum_{k=0}^{M}[6k+1]\frac{(aq,q/a;q^2)_k(q^2,bq,q/b;q^4)_k}{(q^4,q^4/a,aq^4;q^4)_k(q^2/b,bq^2;q^2)_k}q^{2k}
\equiv
 \frac{(q^5,bq^2,q^2/b;q^4)_{(n-1)/4}}{(q,q^4/b,bq^4;q^4)_{(n-1)/4}}.
\label{eq:wei-w}
\end{align}

The $q\to q^2$, $a\to q$, $b\to aq$, $d\to q^{1-n}$, $f\to q^{2}$
case of \eqref{GR} is
\begin{align*}
\sum_{k=0}^{M}[6k+1]\frac{(q^{1-n},q^{1+n};q^2)_k(q^2,bq,q/b;q^4)_k}{(q^4,q^{4+n},q^{4-n};q^4)_k(q^2/b,bq^2;q^2)_k}q^{2k}
=\frac{(q^5,aq^2,q^2/a;q^4)_{(n-1)/4}}{(q,q^4/a,aq^4;q^4)_{(n-1)/4}}.
\end{align*}
 Thus we are led to the result: modulo
$(1-bq^n)(b-q^n)$,
\begin{align}
\sum_{k=0}^{M}[6k+1]\frac{(aq,q/a;q^2)_k(q^2,bq,q/b;q^4)_k}{(q^4,q^4/a,aq^4;q^4)_k(q^2/b,bq^2;q^2)_k}q^{2k}
\equiv
\frac{(q^5,aq^2,q^2/a;q^4)_{(n-1)/4}}{(q,q^4/a,aq^4;q^4)_{(n-1)/4}}.
\label{eq:wei-x}
\end{align}
Using \eqref{Chinese-c}, \eqref{Chinese-d}, and the Chinese
remainder theorem for coprime polynomials, we can derive Theorem
\ref{thm-f} from Lemma \ref{lemm-b}, \eqref{eq:wei-w}, and
\eqref{eq:wei-x}.
\end{proof}

\begin{proof}[Proof of Theorem \ref{thm-d}]
Letting $b\to1$ in Theorem \ref{thm-h}, we obtain the formula:
modulo $[n]\Phi_n(q)^2(1-aq^n)(a-q^n)$,
\begin{align}
&\sum_{k=0}^{M}[6k+1]\frac{(aq,q/a;q^2)_k(q;q^4)_k^2(q^2;q^4)_k}{(q^4/a,aq^4;q^4)_k(q^2;q^2)_k^2(q^4;q^4)_k}q^{2k}
\notag\\[3mm]
 &\quad\equiv[n]\frac{(1-aq^n)(a-q^n)(-2+q^n)}{(1-a)^2}\frac{(aq^2,q^2/a;q^4)_{(n-1)/4}}{(q^4/a,aq^4;q^4)_{(n-1)/4}}
\notag
\end{align}
\begin{align}
&\qquad-[n]\frac{(1-q^n)^2(-1-a^2+aq^n)}{(1-a)^2}\frac{(q^2;q^4)_{(n-1)/4}^2}{(q^4;q^4)_{(n-1)/4}^2}
\notag\\[3mm]
 &\quad=[n](1-q^n)^2\frac{(q^2;q^4)_{(n-1)/4}^2}{(q^4;q^4)_{(n-1)/4}^2}+[n](2-q^n)\Omega(a,q,n),
\label{eq:wei-y}
\end{align}
where
\begin{align*}
\Omega(a,q,n)=
 \frac{a(1-q^n)^2}{(1-a)^2}\frac{(q^2;q^4)_{(n-1)/4}^2}{(q^4;q^4)_{(n-1)/4}^2}
-\frac{(1-aq^n)(a-q^n)}{(1-a)^2}\frac{(aq^2,q^2/a;q^4)_{(n-1)/4}}{(q^4/a,aq^4;q^4)_{(n-1)/4}}.
\end{align*}
 By the L'H\^{o}spital rule, it is ordinary to calculate the limit
\begin{align*}
&\lim_{a\to1}\Omega(a,q,n)\\[5pt]
&\quad=\frac{(q^2;q^4)_{(n-1)/4}^2}{(q^4;q^4)_{(n-1)/4}^2}
\bigg\{q^n+[n]^2\sum_{j=1}^{(n-1)/2}\frac{(-1)^jq^{2j}}{[2j]^2}\bigg\}.
\end{align*}
Letting $a\to1$ in \eqref{eq:wei-y} and utilizing the upper limit,
we get hold of Theorem \ref{thm-d}.
\end{proof}


\end{document}